\documentclass[a4paper,11pt]{article}

\usepackage[left=1.5cm,right=1.5cm,top=1.5cm,bottom=1.5cm]{geometry}
\usepackage{amsmath,amssymb,amsthm,mathtools}
\usepackage{newtxtext,newtxmath}
\DeclareSymbolFont{lightAMSbb}{U}{msb}{m}{n}
\DeclareSymbolFontAlphabet{\mathbb}{lightAMSbb}
\usepackage{enumitem}
\usepackage{indentfirst}
\usepackage{microtype}
\usepackage{xurl}
\usepackage{authblk}
\usepackage{cite}
\usepackage[
  colorlinks=true,
  linkcolor=blue,
  citecolor=blue,
  urlcolor=blue
]{hyperref}

\hypersetup{
  pdftitle={A Necessary and Sufficient Condition for the Existence of Global Smooth Solutions to the One-Dimensional Inviscid Burgers-Vlasov System},
  pdfauthor={Chengfeng Ren and Huimin Yu},
  pdfsubject={Global smooth solutions for the one-dimensional inviscid Burgers-Vlasov system},
  pdfkeywords={Burgers-Vlasov system, global smooth solutions, necessary and sufficient condition, critical threshold, finite-time blow-up}
}

\allowdisplaybreaks[3]
\numberwithin{equation}{section}
\newtheorem{theorem}{Theorem}[section]
\newtheorem{proposition}[theorem]{Proposition}

\newtheorem{corollary}[theorem]{Corollary}
\theoremstyle{remark}

\newcommand{\eqtaglabel}[1]{\refstepcounter{equation}\tag{\theequation}\label{#1}}
\newcommand{\thmref}[1]{Theorem~\ref{#1}}
\newcommand{\propref}[1]{Proposition~\ref{#1}}

\newcommand{\corref}[1]{Corollary~\ref{#1}}
\newcommand{\secref}[1]{Section~\ref{#1}}

\title{\bfseries A Necessary and Sufficient Condition for the Existence of Global Smooth Solutions to the One-Dimensional Inviscid Burgers-Vlasov System}
\author{Chengfeng Ren\qquad Huimin Yu\textsuperscript{*}}
\affil{\small\itshape School of Mathematics and Statistics, Shandong Normal University, Jinan 250358, China}
\date{}

\begin{document}

\maketitle
\vspace{-1.5em}
\begingroup
\renewcommand{\thefootnote}{}
\begin{NoHyper}
\footnotetext{E-mail: \href{mailto:15163322620@163.com}{15163322620@163.com}}
\footnotetext{\textsuperscript{*}Corresponding author: \href{mailto:hmyu@sdnu.edu.cn}{hmyu@sdnu.edu.cn}}
\end{NoHyper}
\endgroup

\begin{center}
\begin{minipage}{\textwidth}
\small
\noindent\textbf{Abstract}: We investigate global smooth solvability for the one-dimensional inviscid Burgers-Vlasov system. For a smooth initial fluid velocity and a nonnegative, compactly supported particle distribution, we establish a sharp necessary and sufficient threshold expressed through the sum of the initial velocity gradient and the initial particle density. The principal difficulty is to control simultaneously the nonlinear steepening of the fluid velocity, the drag coupling, the particle velocity support, and the full hierarchy of derivatives. To overcome it, we introduce a combined gradient-density quantity whose evolution closes along fluid characteristics. Its sign invariance, together with characteristic support estimates and fixed-order derivative bounds, yields global continuation whenever the initial threshold is nonnegative. If the threshold fails at one point, a Riccati comparison gives an explicit finite upper bound for the maximal smooth lifespan and identifies negative velocity-gradient blow-up when this bound is attained. We also use nonnegative regularization and compactness to extend the threshold result to initial data of any fixed finite regularity.\par
\noindent\textbf{Keywords}: Burgers-Vlasov system, global smooth solutions, necessary and sufficient condition, critical threshold, finite-time blow-up.\par
\noindent\textbf{Mathematics Subject Classification 2020}: Primary 35Q83; Secondary 35L60, 35B44.
\end{minipage}
\end{center}

\section{Introduction}\label{sec:introduction}

We consider the following one-dimensional inviscid Burgers-Vlasov system:
\[
\begin{cases}
u_t+uu_x=\displaystyle\int_{\mathbb R}vf\,dv
-u\displaystyle\int_{\mathbb R}f\,dv,\\[1mm]
f_t+vf_x+\bigl(f(u-v)\bigr)_v=0,
\end{cases}
\eqtaglabel{eq:1.1}
\]
\par\medskip
\noindent with the initial data
\[
u(0,x)=u_0(x),\qquad
f(0,x,v)=f_0(x,v)\ge0.
\eqtaglabel{eq:1.2}
\]
Here \(u=u(t,x)\) denotes the fluid velocity and \(f=f(t,x,v)\) denotes the particle distribution function in phase space.

The Burgers-Vlasov system is a simplified kinetic-fluid model for dispersed two-phase flows. The carrier fluid is described by a macroscopic velocity field, whereas the dispersed particles are represented by a distribution function in phase space. The two phases interact through the drag force \(u-v\), which tends to align the particle velocity with the surrounding fluid velocity. Related kinetic descriptions of dispersed and bubbly flows, and their macroscopic limits, were studied in \cite{ref12} and \cite{ref13}, respectively.

For viscous Burgers-Vlasov systems, Domelevo and Roquejoffre \cite{ref5} established the existence of classical solutions and studied the stability of travelling waves. Domelevo \cite{ref6} subsequently developed a well-posedness theory for less regular fluid data and measure-valued particle distributions. For more general kinetic-fluid spray models, Baranger and Desvillettes \cite{ref10} proved local existence and uniqueness of classical solutions under suitable smoothness and support assumptions.

In the inviscid setting, under the assumptions \(u_0\in L^\infty(\mathbb R)\) and \((1+v^2)f_0\in L^1(\mathbb R^2)\), Yu and Cao \cite{ref1} proved the existence of at least one global \(L^\infty-L^1\) weak solution to \eqref{eq:1.1}--\eqref{eq:1.2}, in which \(u\) is an entropy solution of the Burgers equation with a source term and \(f\) is a distributional solution of the Vlasov equation. Their proof uses an artificial-viscosity approximation, \(L^\infty\) compensated compactness, and Dunford-Pettis weak compactness to pass to the inviscid limit. For the same Burgers-Vlasov model, Cao and Wang \cite{ref2} studied the vanishing-viscosity limit from the viscous system and obtained a finite-energy weak solution of the inviscid equations. Global weak solutions have also been established for related Euler-Vlasov systems \cite{ref3,ref4}. These results provide a weak-solution theory after possible shock formation, but they do not characterize which smooth initial data prevent the formation of singularities.

The distinction between global smoothness and finite-time breakdown is closely related to the critical-threshold phenomenon initiated for Euler-Poisson equations \cite{ref14} and subsequently developed for Euler equations with alignment and other nonlocal forces \cite{ref7,ref8,ref9}. In these models, a suitable combination of the velocity gradient and the density determines whether the solution remains smooth or develops a singularity. Although those models differ from \eqref{eq:1.1}, their threshold analyses provide a natural background for the present problem and motivate the search for a corresponding sharp threshold for the inviscid Burgers-Vlasov system.

The critical-threshold results cited above are generally formulated at a fixed finite regularity. More specifically, Bhatnagar, Liu, and Tan \cite{ref9} formulate their local theory in finite Sobolev spaces, Carrillo et al.\ \cite{ref8} and Tadmor and Tan \cite{ref7} use strong or Lipschitz solution classes, and the one-dimensional theory of Engelberg, Liu, and Tadmor \cite{ref14} is stated for \(C^1\) initial data. In the present paper, we first work with smooth data so that the characteristic identities and the derivative hierarchy can be derived without additional approximation notation. The estimates themselves are obtained separately at each finite order. We then use nonnegative regularization and compactness to extend the threshold result to data of any prescribed finite regularity. Thus the assumption of infinitely differentiable data in the main theorem is a convenient framework for the smooth statement, rather than a minimal regularity requirement.

Our purpose is to obtain a necessary and sufficient condition for the existence of a global smooth solution to \eqref{eq:1.1}--\eqref{eq:1.2}. Define
\[
\rho_0(x)=\int_{\mathbb R}f_0(x,v)\,dv
\]
and set \(F_0=u_0'+\rho_0\). The main theorem states that a global smooth solution exists precisely when \(F_0(x)\ge0\) for every \(x\in\mathbb R\). The proof is based on the combined quantity \(F=u_x+\rho\), which obeys
\(F_t+uF_x=F(\rho-F)\) along fluid characteristics. If \(F_0\ge0\), this equation preserves nonnegativity and, together with estimates for the particle velocity support, the distribution function, and higher derivatives, yields a priori bounds on every finite time interval. If \(F_0(x_0)<0\) at some point \(x_0\), then along the fluid characteristic issued from \(x_0\),
\(\frac{d}{dt}(-F)\ge(-F)^2\), and hence
\[
T_{\text{max}}\le\dfrac{1}{-F_0(x_0)}<\infty.
\]
Thus the maximal smooth lifespan is finite and a global smooth solution cannot exist.

The remainder of the paper is organized as follows. \secref{sec:local} introduces the function spaces, states the main theorem, and establishes the local solution theory and continuation criterion used in the proof. \secref{sec:key} derives the particle characteristic formula and the evolution equation for \(F\). \secref{sec:main-proof} proves the main theorem.

\section{Main Result and Local Existence}\label{sec:local}

For an integer \(k\ge0\), let
\[
C_b^k(\mathbb R^d)
=
\left\{
g\in C^k(\mathbb R^d):
\partial^\alpha g\in L^\infty(\mathbb R^d)
\ \text{for all }|\alpha|\le k
\right\},
\qquad
\|g\|_{C_b^k}
=
\sum_{|\alpha|\le k}\|\partial^\alpha g\|_{L^\infty},
\]
and define
\[
C_b^\infty(\mathbb R^d)
=\bigcap_{k=0}^{\infty}C_b^k(\mathbb R^d).
\]
We write \(C_c^k(\mathbb R^d)\) for the space of compactly supported \(C^k\) functions.

For \(d\in\{1,2\}\) and an integer \(m\ge0\), we also use the Sobolev space
\[
W^{m,\infty}(\mathbb R^d)
=
\left\{
g\in L^\infty(\mathbb R^d):
\partial^\alpha g\in L^\infty(\mathbb R^d)
\ \text{for all }|\alpha|\le m
\right\},
\]
where the derivatives are understood in the distributional sense and
\[
\|g\|_{W^{m,\infty}}
=\sum_{|\alpha|\le m}\|\partial^\alpha g\|_{L^\infty}.
\]
Here \(d=1\) is used for the fluid velocity \(u\), whereas \(d=2\) is used for the particle distribution \(f\). If \(m\ge2\), \(g\in W^{m,\infty}(\mathbb R^d)\), and \(|\alpha|\le m-1\), then \(\partial^\alpha g\in W^{1,\infty}(\mathbb R^d)\) and hence admits a Lipschitz-continuous representative. In particular, all derivatives of order at most \(m-1\) are continuous.

Having introduced the function spaces, we now state the main result. It gives a necessary and sufficient condition for problem \eqref{eq:1.1}--\eqref{eq:1.2} to admit a global smooth solution and an upper bound for the maximal lifespan when this condition fails. The local smooth solution and the maximal lifespan appearing in the statement are provided by \propref{prop:local} below.

\begin{theorem}[Main theorem]\label{thm:main}
Assume that
\[
u_0\in C_b^\infty(\mathbb R),\qquad
0\le f_0\in C_c^\infty(\mathbb R^2).
\eqtaglabel{eq:2.1}
\]
Define
\[
\rho_0(x)=\int_{\mathbb R}f_0(x,v)\,dv,
\qquad
F_0(x)=u_0'(x)+\rho_0(x).
\]
Then problem \eqref{eq:1.1}--\eqref{eq:1.2} admits a global smooth solution if and only if
\[
F_0(x)\ge0,
\qquad x\in\mathbb R.
\eqtaglabel{eq:2.2}
\]
Moreover, if \(F_0(x_0)<0\) for some \(x_0\in\mathbb R\), then the maximal lifespan \(T_{\text{max}}\) of the corresponding local smooth solution satisfies
\[
T_{\text{max}}\le\dfrac{1}{-F_0(x_0)}<\infty.
\eqtaglabel{eq:2.3}
\]
\end{theorem}

For each fixed \(t\), we use the abbreviations
\[
u(t)=u(t,\cdot),
\qquad
f(t)=f(t,\cdot,\cdot).
\]
Accordingly, \(\operatorname{supp}f(t)\) denotes the support of the function \((x,v)\mapsto f(t,x,v)\) in phase space. The proof of \thmref{thm:main} uses the existence and uniqueness of a local smooth solution together with a continuation criterion. We establish these facts in the following proposition.

\begin{proposition}[Local smooth solutions and continuation criterion]\label{prop:local}
Under assumption \eqref{eq:2.1}, there exists \(T_0>0\) such that \eqref{eq:1.1}--\eqref{eq:1.2} admits a unique local smooth solution. Moreover, for every integer \(m\ge2\),
\[
\begin{aligned}
u&\in C([0,T_0];C_b^m(\mathbb R))
\cap C^1([0,T_0];C_b^{m-1}(\mathbb R)),\\
f&\in C([0,T_0];C_b^m(\mathbb R^2))
\cap C^1([0,T_0];C_b^{m-1}(\mathbb R^2)),
\end{aligned}
\]
and the phase-space support of \(f(t)\) is contained in a fixed compact set for \(0\le t\le T_0\). Let \([0,T_{\text{max}})\) be the maximal interval of smooth existence. If \(T_{\text{max}}<\infty\), then
\[
\limsup_{t\to T_{\text{max}}^{-}}
\left(
\|u(t)\|_{C_b^2}
+\|f(t)\|_{C_b^2}
\right)=+\infty.
\eqtaglabel{eq:2.4}
\]
\end{proposition}

\begin{proof}
We adapt the characteristic iteration used in \cite{ref10} for a coupled Euler-Vlasov spray model to the present one-dimensional Burgers-Vlasov system and carry out the estimates in \(C_b^k\) spaces. Throughout the proof, \(C\) denotes a positive constant depending only on the initial-data norms and the support of \(f_0\); it is independent of \(n\) and may vary from line to line. Since \(f_0\) is compactly supported, we may choose \(R_0>0\) sufficiently large so that
\[
\|u_0\|_{L^\infty}\le R_0,
\qquad
\operatorname{supp}f_0
\subset
\{(x,v)\in\mathbb R^2:|x|+|v|\le R_0\}.
\]
Set \(u^{(0)}=u_0\) and \(f^{(0)}=f_0\). For \(n\ge1\), all iterates are assigned the common initial data
\[
u^{(n)}(0,x)=u_0(x),\qquad
f^{(n)}(0,x,v)=f_0(x,v).
\]
Once \(u^{(n-1)}\) is known, first solve
\[
f_t^{(n)}+vf_x^{(n)}
+\bigl((u^{(n-1)}-v)f^{(n)}\bigr)_v=0,
\eqtaglabel{eq:2.5}
\]
or, equivalently,
\[
f_t^{(n)}+vf_x^{(n)}
+\bigl(u^{(n-1)}-v\bigr)f_v^{(n)}=f^{(n)}.
\]
Set
\[
\rho^{(n)}(t,x)=\int_{\mathbb R}f^{(n)}(t,x,v)\,dv,
\qquad
j^{(n)}(t,x)=\int_{\mathbb R}v f^{(n)}(t,x,v)\,dv,
\]
and then solve
\[
\begin{aligned}
u_t^{(n)}+u^{(n-1)}u_x^{(n)}
={}&j^{(n)}-\rho^{(n)}u^{(n)}.
\end{aligned}
\eqtaglabel{eq:2.6}
\]
The associated particle characteristics satisfy
\[
\begin{cases}
\displaystyle\frac{dX^{(n)}}{dt}=V^{(n)},\\[2mm]
\displaystyle\frac{dV^{(n)}}{dt}=u^{(n-1)}(t,X^{(n)})-V^{(n)},
\end{cases}
\qquad
\bigl(X^{(n)}(0),V^{(n)}(0)\bigr)=(x_0,v_0),
\eqtaglabel{eq:2.7}
\]
and therefore
\[
f^{(n)}(t,X^{(n)}(t),V^{(n)}(t))=e^tf_0(x_0,v_0),
\eqtaglabel{eq:2.8}
\]
\[
\begin{aligned}
V^{(n)}(t)
&=e^{-t}v_0+
\int_0^t e^{-(t-\tau)}u^{(n-1)}(\tau,X^{(n)}(\tau))\,d\tau,\\
X^{(n)}(t)
&=x_0+\int_0^tV^{(n)}(\tau)\,d\tau.
\end{aligned}
\eqtaglabel{eq:2.9}
\]

We next establish the support and second-order estimates used to close the iteration. Fix \(M>0\) sufficiently large so that
\[
\|u_0\|_{C_b^2}+\|f_0\|_{C_b^2}<\frac{M}{2}.
\]
Since only a local solution is being constructed, we may restrict attention to \(0<T\le1\). As the induction hypothesis, assume that
\[
\sup_{0\le t\le T}\|u^{(n-1)}(t)\|_{C_b^2}\le M.
\]
The induction starts from \(u^{(0)}=u_0\), and the estimates below close it for sufficiently small \(T\). For every characteristic issued from \(\operatorname{supp}f_0\), the first identity in \eqref{eq:2.9} gives, for \(0\le t\le T\le1\),
\[
\begin{aligned}
|V^{(n)}(t)|
&\le e^{-t}|v_0|
+\int_0^t e^{-(t-\tau)}
\|u^{(n-1)}(\tau)\|_{L^\infty}\,d\tau\\
&\le R_0+M.
\end{aligned}
\]
The second identity in \eqref{eq:2.9} then yields
\[
\begin{aligned}
|X^{(n)}(t)|
&\le |x_0|+\int_0^t|V^{(n)}(\tau)|\,d\tau\\
&\le R_0+T(R_0+M)
\le2R_0+M.
\end{aligned}
\]
Consequently,
\[
|X^{(n)}(t)|+|V^{(n)}(t)|\le3R_0+2M.
\]
Setting \(R=3R_0+2M\), we obtain the \(n\)-independent support bound
\[
\operatorname{supp}f^{(n)}(t)
\subset
\{(x,v)\in\mathbb R^2:|x|+|v|\le R\},
\qquad 0\le t\le T.
\]

This support estimate also controls the velocity moments in \eqref{eq:2.6}. For \(0\le k\le2\), differentiation under the integral sign gives
\[
\begin{aligned}
\|\partial_x^k\rho^{(n)}(t)\|_{L_x^\infty}
&\le\int_{-R}^{R}
\|\partial_x^kf^{(n)}(t,\cdot,v)\|_{L_x^\infty}\,dv\\
&\le2R\|\partial_x^kf^{(n)}(t)\|_{L_{x,v}^\infty},
\end{aligned}
\]
and
\[
\begin{aligned}
\|\partial_x^kj^{(n)}(t)\|_{L_x^\infty}
&\le\int_{-R}^{R}|v|\,
\|\partial_x^kf^{(n)}(t,\cdot,v)\|_{L_x^\infty}\,dv\\
&\le R^2\|\partial_x^kf^{(n)}(t)\|_{L_{x,v}^\infty}.
\end{aligned}
\]
Consequently,
\[
\|\rho^{(n)}(t)\|_{C_b^2(\mathbb R)}
\le2R\|f^{(n)}(t)\|_{C_b^2(\mathbb R^2)},
\]
and
\[
\|j^{(n)}(t)\|_{C_b^2(\mathbb R)}
\le R^2\|f^{(n)}(t)\|_{C_b^2(\mathbb R^2)}.
\]
Adding these two estimates, we obtain
\[
\begin{aligned}
\|\rho^{(n)}(t)\|_{C_b^2}
+\|j^{(n)}(t)\|_{C_b^2}
&\le(2R+R^2)\|f^{(n)}(t)\|_{C_b^2}\\
&\le C_R\|f^{(n)}(t)\|_{C_b^2}.
\end{aligned}
\]

We next derive the second-order estimate for \(f^{(n)}\). Differentiating the nondivergence form of \eqref{eq:2.5} once with respect to \(x\) and \(v\), respectively, we obtain
\[
\begin{aligned}
(f_x^{(n)})_t+v(f_x^{(n)})_x
+\bigl(u^{(n-1)}-v\bigr)(f_x^{(n)})_v
&=f_x^{(n)}-u_x^{(n-1)}f_v^{(n)},\\
(f_v^{(n)})_t+v(f_v^{(n)})_x
+\bigl(u^{(n-1)}-v\bigr)(f_v^{(n)})_v
&=2f_v^{(n)}-f_x^{(n)}.
\end{aligned}
\]
The second-order derivatives satisfy
\[
\begin{aligned}
(f_{xx}^{(n)})_t+v(f_{xx}^{(n)})_x
+\bigl(u^{(n-1)}-v\bigr)(f_{xx}^{(n)})_v
={}&f_{xx}^{(n)}-2u_x^{(n-1)}f_{xv}^{(n)}\\
&\quad-u_{xx}^{(n-1)}f_v^{(n)},\\[1mm]
(f_{xv}^{(n)})_t+v(f_{xv}^{(n)})_x
+\bigl(u^{(n-1)}-v\bigr)(f_{xv}^{(n)})_v
&=2f_{xv}^{(n)}-f_{xx}^{(n)}-u_x^{(n-1)}f_{vv}^{(n)},\\[1mm]
(f_{vv}^{(n)})_t+v(f_{vv}^{(n)})_x
+\bigl(u^{(n-1)}-v\bigr)(f_{vv}^{(n)})_v
&=3f_{vv}^{(n)}-2f_{xv}^{(n)}.
\end{aligned}
\]
Thus the right-hand sides contain only derivatives of \(f^{(n)}\) of order at most two, derivatives of \(u^{(n-1)}\) of order at most two, and their products. Integrating these equations along the particle characteristics, taking the \(L^\infty\)-norm, and summing over all derivatives of order at most two give
\[
\begin{aligned}
\|f^{(n)}(t)\|_{C_b^2}
\le{}&\|f_0\|_{C_b^2}\\
&+C\int_0^t\Bigl[
\|f^{(n)}(\tau)\|_{C_b^2}
+\|u_x^{(n-1)}(\tau)\|_{L^\infty}
 \|f^{(n)}(\tau)\|_{C_b^2}\\
&\hspace{32mm}
+\|u_{xx}^{(n-1)}(\tau)\|_{L^\infty}
 \|f_v^{(n)}(\tau)\|_{L^\infty}
\Bigr]\,d\tau\\
\le{}&\|f_0\|_{C_b^2}
+C\int_0^t
\bigl(1+\|u^{(n-1)}(\tau)\|_{C_b^2}\bigr)
\|f^{(n)}(\tau)\|_{C_b^2}\,d\tau.
\end{aligned}
\eqtaglabel{eq:2.10}
\]

We now derive the corresponding estimate for \(u^{(n)}\). Differentiating \eqref{eq:2.6} once with respect to \(x\) gives
\[
\begin{aligned}
(u_x^{(n)})_t+u^{(n-1)}(u_x^{(n)})_x
={}&j_x^{(n)}-u^{(n)}\rho_x^{(n)}\\
&-\bigl(\rho^{(n)}+u_x^{(n-1)}\bigr)u_x^{(n)},
\end{aligned}
\]
while a second differentiation gives
\[
\begin{aligned}
(u_{xx}^{(n)})_t+u^{(n-1)}(u_{xx}^{(n)})_x
={}&j_{xx}^{(n)}-u^{(n)}\rho_{xx}^{(n)}
-2\rho_x^{(n)}u_x^{(n)}\\
&-\bigl(\rho^{(n)}+2u_x^{(n-1)}\bigr)u_{xx}^{(n)}
-u_{xx}^{(n-1)}u_x^{(n)}.
\end{aligned}
\]
Integrating the original equation and these two differentiated equations along the fluid characteristics, taking the \(L^\infty\)-norm, and summing the estimates through order two, we obtain
\[
\begin{aligned}
\|u^{(n)}(t)\|_{C_b^2}
\le{}&\|u_0\|_{C_b^2}\\
&+C\int_0^t\Bigl[
\|j^{(n)}(\tau)\|_{C_b^2}
+\|\rho^{(n)}(\tau)\|_{C_b^2}
 \bigl(1+\|u^{(n)}(\tau)\|_{C_b^2}\bigr)\\
&\hspace{27mm}
+\|u^{(n-1)}(\tau)\|_{C_b^2}
 \bigl(1+\|u^{(n)}(\tau)\|_{C_b^2}\bigr)
\Bigr]\,d\tau\\
\le{}&\|u_0\|_{C_b^2}
+C\int_0^t
\bigl(1+\|u^{(n-1)}(\tau)\|_{C_b^2}
+\|f^{(n)}(\tau)\|_{C_b^2}\bigr)\\
&\qquad\quad\times
\bigl(1+\|u^{(n)}(\tau)\|_{C_b^2}\bigr)\,d\tau.
\end{aligned}
\eqtaglabel{eq:2.11}
\]
Under the induction hypothesis, \eqref{eq:2.10} and Grönwall's inequality give
\[
\|f^{(n)}(t)\|_{C_b^2}
\le \|f_0\|_{C_b^2}e^{C(1+M)t}.
\]
We may therefore choose \(0<T_0\le1\) sufficiently small so that
\[
\sup_{0\le t\le T_0}\|f^{(n)}(t)\|_{C_b^2}
\le\frac{M}{2}.
\]
Substituting this bound into \eqref{eq:2.11} and applying Grönwall's inequality again yield
\[
1+\|u^{(n)}(t)\|_{C_b^2}
\le
\left(1+\|u_0\|_{C_b^2}\right)
e^{C(1+\frac{3M}{2})t}.
\]
After possibly reducing \(T_0\) further, we obtain
\[
\sup_{0\le t\le T_0}\|u^{(n)}(t)\|_{C_b^2}
\le\frac{M}{2}.
\]
Consequently, the induction closes and
\[
\sup_{0\le t\le T_0}
\left(
\|u^{(n)}(t)\|_{C_b^2}+\|f^{(n)}(t)\|_{C_b^2}
\right)\le M.
\eqtaglabel{eq:2.12}
\]
Here \(M\) and \(T_0\) depend only on the initial \(C_b^2\)-norms and the initial support, and are independent of \(n\).

For \(n\ge1\), set
\[
\delta u^{(n)}=u^{(n)}-u^{(n-1)},\qquad
\delta f^{(n)}=f^{(n)}-f^{(n-1)}.
\]
For \(n\ge2\), subtracting two successive iterative equations gives
\[
(\delta f^{(n)})_t
+v(\delta f^{(n)})_x
+\bigl(u^{(n-1)}-v\bigr)(\delta f^{(n)})_v
=\delta f^{(n)}-\delta u^{(n-1)}f_v^{(n-1)},
\]
and
\[
\begin{aligned}
(\delta u^{(n)})_t
+u^{(n-1)}(\delta u^{(n)})_x
={}&\delta j^{(n)}-\rho^{(n)}\delta u^{(n)}
-u^{(n-1)}\delta\rho^{(n)}\\
&-\delta u^{(n-1)}u_x^{(n-1)}.
\end{aligned}
\]
Both differences vanish initially:
\[
\delta u^{(n)}(0,x)=0,
\qquad
\delta f^{(n)}(0,x,v)=0.
\]
Moreover,
\[
\delta\rho^{(n)}
=\int_{\mathbb R}\delta f^{(n)}\,dv,\qquad
\delta j^{(n)}
=\int_{\mathbb R}v\delta f^{(n)}\,dv.
\]
To justify the contraction estimate, we differentiate the particle difference equation once. Differentiation with respect to \(x\) gives
\[
\begin{aligned}
&\bigl((\delta f^{(n)})_x\bigr)_t
+v\bigl((\delta f^{(n)})_x\bigr)_x
+\bigl(u^{(n-1)}-v\bigr)\bigl((\delta f^{(n)})_x\bigr)_v\\
&\quad=(\delta f^{(n)})_x
-u_x^{(n-1)}(\delta f^{(n)})_v
-(\delta u^{(n-1)})_x f_v^{(n-1)}
-\delta u^{(n-1)}f_{xv}^{(n-1)},
\end{aligned}
\]
while differentiation with respect to \(v\) yields
\[
\begin{aligned}
&\bigl((\delta f^{(n)})_v\bigr)_t
+v\bigl((\delta f^{(n)})_v\bigr)_x
+\bigl(u^{(n-1)}-v\bigr)\bigl((\delta f^{(n)})_v\bigr)_v\\
&\quad=2(\delta f^{(n)})_v
-(\delta f^{(n)})_x
-\delta u^{(n-1)}f_{vv}^{(n-1)}.
\end{aligned}
\]
The zero initial data and the uniform estimate \eqref{eq:2.12} therefore imply, after integration along the particle characteristics,
\[
\|\delta f^{(n)}(t)\|_{C_b^1}
\le C\int_0^t
\left(
\|\delta f^{(n)}(\tau)\|_{C_b^1}
+\|\delta u^{(n-1)}(\tau)\|_{C_b^1}
\right)\,d\tau.
\]
The velocity-moment calculation used above, now applied to \(\delta f^{(n)}\), gives
\[
\|\delta\rho^{(n)}(t)\|_{C_b^1}
+\|\delta j^{(n)}(t)\|_{C_b^1}
\le C_R\|\delta f^{(n)}(t)\|_{C_b^1}.
\]
Differentiating the fluid difference equation with respect to \(x\), we find
\[
\begin{aligned}
&\bigl((\delta u^{(n)})_x\bigr)_t
+u^{(n-1)}\bigl((\delta u^{(n)})_x\bigr)_x\\
&\quad=(\delta j^{(n)})_x
-\rho_x^{(n)}\delta u^{(n)}
-\bigl(\rho^{(n)}+u_x^{(n-1)}\bigr)(\delta u^{(n)})_x\\
&\qquad-u^{(n-1)}(\delta\rho^{(n)})_x
-u_x^{(n-1)}\delta\rho^{(n)}\\
&\qquad-(\delta u^{(n-1)})_x u_x^{(n-1)}
-\delta u^{(n-1)}u_{xx}^{(n-1)}.
\end{aligned}
\]
Using the preceding moment estimate and \eqref{eq:2.12}, integration along the fluid characteristics yields
\[
\begin{aligned}
\|\delta u^{(n)}(t)\|_{C_b^1}
\le C\int_0^t\Bigl(
&\|\delta u^{(n)}(\tau)\|_{C_b^1}
+\|\delta f^{(n)}(\tau)\|_{C_b^1}\\
&+\|\delta u^{(n-1)}(\tau)\|_{C_b^1}
\Bigr)\,d\tau.
\end{aligned}
\]
Set
\[
\mathcal D_n(t)
:=
\sup_{0\le s\le t}
\left(
\|\delta u^{(n)}(s)\|_{C_b^1}
+\|\delta f^{(n)}(s)\|_{C_b^1}
\right).
\]
Combining the last two estimates gives
\[
\mathcal D_n(t)
\le C\int_0^t
\bigl(\mathcal D_n(\tau)+\mathcal D_{n-1}(\tau)\bigr)\,d\tau.
\]
The integral form of Grönwall's inequality then gives, for \(0\le t\le T\le T_0\),
\[
\begin{aligned}
\mathcal D_n(t)
&\le C\int_0^t e^{C(t-\tau)}\mathcal D_{n-1}(\tau)\,d\tau\\
&\le Cte^{Ct}\mathcal D_{n-1}(T).
\end{aligned}
\]
Taking \(t=T\), we obtain
\[
\mathcal D_n(T)
\le CTe^{CT}\mathcal D_{n-1}(T).
\eqtaglabel{eq:2.13}
\]
Reducing \(T_0\) once more if necessary, we may assume that
\[
CT_0e^{CT_0}\le\frac12.
\]
Then \eqref{eq:2.13} implies
\[
\mathcal D_n(T_0)\le\frac12\mathcal D_{n-1}(T_0)
\le2^{-(n-1)}\mathcal D_1(T_0).
\]
Consequently, for \(m>n\),
\[
\begin{aligned}
&\sup_{0\le t\le T_0}
\left(
\|u^{(m)}(t)-u^{(n)}(t)\|_{C_b^1}
+\|f^{(m)}(t)-f^{(n)}(t)\|_{C_b^1}
\right)\\
&\qquad\le
\sum_{k=n+1}^{m}\mathcal D_k(T_0)
\le 2^{1-n}\mathcal D_1(T_0)\longrightarrow0
\qquad(n\to\infty).
\end{aligned}
\]
Thus \((u^{(n)},f^{(n)})\) is Cauchy in \(C([0,T_0];C_b^1(\mathbb R)\times C_b^1(\mathbb R^2))\).

To upgrade this convergence to \(C_b^2\), it remains to establish a uniform third-order bound. For \(p+q=3\), differentiating the particle equation gives
\[
\begin{aligned}
&\bigl(\partial_x^p\partial_v^q f^{(n)}\bigr)_t
+v\bigl(\partial_x^p\partial_v^q f^{(n)}\bigr)_x
+\bigl(u^{(n-1)}-v\bigr)
 \bigl(\partial_x^p\partial_v^q f^{(n)}\bigr)_v\\
&\quad=(q+1)\partial_x^p\partial_v^q f^{(n)}
-q\partial_x^{p+1}\partial_v^{q-1}f^{(n)}
-\sum_{h=1}^{p}\binom ph
\bigl(\partial_x^hu^{(n-1)}\bigr)
\partial_x^{p-h}\partial_v^{q+1}f^{(n)},
\end{aligned}
\]
where the term containing \(q-1\) is omitted when \(q=0\). Since \(p+q=3\), the derivative of \(f^{(n)}\) in the \(h\)-th product has order
\[
(p-h)+(q+1)=4-h.
\]
Hence, using the uniform second-order bound \eqref{eq:2.12}, we obtain
\[
\begin{aligned}
&\left\|
\sum_{h=1}^{p}\binom ph
\bigl(\partial_x^hu^{(n-1)}\bigr)
\partial_x^{p-h}\partial_v^{q+1}f^{(n)}
\right\|_{L^\infty}
\\
&\quad\le
C\|u_x^{(n-1)}\|_{L^\infty}\|f^{(n)}\|_{C_b^3}
+C\|u_{xx}^{(n-1)}\|_{L^\infty}\|f^{(n)}\|_{C_b^2}
\\
&\qquad\quad
+\|u_{xxx}^{(n-1)}\|_{L^\infty}\|f_v^{(n)}\|_{L^\infty}
\\
&\quad\le
C\left(
\|f^{(n)}\|_{C_b^3}
+\|u^{(n-1)}\|_{C_b^3}
\right).
\end{aligned}
\]

For the fluid equation, differentiating three times in \(x\) gives
\[
\begin{aligned}
(u_{xxx}^{(n)})_t+u^{(n-1)}(u_{xxx}^{(n)})_x
={}&j_{xxx}^{(n)}-u^{(n)}\rho_{xxx}^{(n)}
-3u_x^{(n)}\rho_{xx}^{(n)}
-3u_{xx}^{(n)}\rho_x^{(n)}\\
&-\bigl(\rho^{(n)}+3u_x^{(n-1)}\bigr)u_{xxx}^{(n)}
-3u_{xx}^{(n-1)}u_{xx}^{(n)}
-u_{xxx}^{(n-1)}u_x^{(n)}.
\end{aligned}
\]
The velocity-moment calculation used above, now at order three, gives
\[
\|\rho^{(n)}(t)\|_{C_b^3}
+\|j^{(n)}(t)\|_{C_b^3}
\le C_R\|f^{(n)}(t)\|_{C_b^3}.
\]
Using this estimate and \eqref{eq:2.12}, and then integrating the differentiated equations along the corresponding characteristics, we obtain
\[
\begin{aligned}
&\|u^{(n)}(t)\|_{C_b^3}
+\|f^{(n)}(t)\|_{C_b^3}\\
&\quad\le
\|u_0\|_{C_b^3}+\|f_0\|_{C_b^3}
+C\int_0^t\Bigl(
1+\|u^{(n)}(\tau)\|_{C_b^3}
+\|f^{(n)}(\tau)\|_{C_b^3}
+\|u^{(n-1)}(\tau)\|_{C_b^3}
\Bigr)\,d\tau .
\end{aligned}
\]
Taking the maximum over the first \(N\) iterates and applying Grönwall's inequality gives a bound independent of \(N\). Letting \(N\to\infty\), we obtain
\[
\sup_{n\ge0}\sup_{0\le t\le T_0}
\left(
\|u^{(n)}(t)\|_{C_b^3}
+\|f^{(n)}(t)\|_{C_b^3}
\right)<\infty.
\eqtaglabel{eq:2.14}
\]

Since \(C_b^3\subset W^{3,\infty}\), the \(L^\infty\)-endpoint Gagliardo-Nirenberg interpolation inequality \cite{ref11} applies on both \(\mathbb R\) and \(\mathbb R^2\), with constants depending only on the dimension. Consequently,
\[
\begin{aligned}
&\|\delta u^{(n)}\|_{C_b^2}
+\|\delta f^{(n)}\|_{C_b^2}\\
&\quad\le C
\left(
\|\delta u^{(n)}\|_{C_b^1}
+\|\delta f^{(n)}\|_{C_b^1}
\right)^{1/2}
\left(
\|\delta u^{(n)}\|_{C_b^3}
+\|\delta f^{(n)}\|_{C_b^3}
\right)^{1/2}.
\end{aligned}
\eqtaglabel{eq:2.15}
\]
By \eqref{eq:2.13}--\eqref{eq:2.15}, the iterates converge in \(C_b^2\). Repeating the third-order estimate at any fixed higher order gives a uniform bound on the same interval \([0,T_0]\); interpolation with the \(C_b^2\)-convergence then yields convergence in \(C_b^m\) for every finite \(m\ge3\). Hence the limit is a local smooth solution. The equations give the stated time regularity, and the same difference estimate yields uniqueness.

It remains to prove the continuation criterion. Suppose, to the contrary, that \(T_{\text{max}}<\infty\) and that there exists \(K<\infty\) such that
\[
\sup_{0\le t<T_{\text{max}}}
\left(\|u(t)\|_{C_b^2}+\|f(t)\|_{C_b^2}\right)\le K.
\]
The same characteristic estimate as in \eqref{eq:2.9}, now applied to the limiting solution, gives a uniform bound on the particle velocity support for \(0\le t<T_{\text{max}}\). Since \(T_{\text{max}}<\infty\) and \(X'=V\), the spatial support is also contained in a fixed compact set. Consequently, for every \(t_0<T_{\text{max}}\), the data \((u(t_0),f(t_0))\) have the same \(C_b^2\) bound and the same support bound. The local construction therefore supplies a number \(\delta>0\), depending only on these uniform bounds, such that the solution restarted at any \(t_0<T_{\text{max}}\) exists at least up to \(t_0+\delta\). Choosing \(t_0>T_{\text{max}}-\frac{\delta}{2}\) produces a classical continuation beyond \(T_{\text{max}}\). The same fixed-order differentiation argument used after \eqref{eq:2.15} shows that this continuation remains smooth. This contradicts the maximality of \([0,T_{\text{max}})\) and proves \eqref{eq:2.4}.
\end{proof}

The preceding construction is carried out for smooth data because \thmref{thm:main} concerns smooth solutions. Its estimates, however, are fixed-order estimates. Combining them with a nonnegative regularization of the initial data gives the following finite-regularity counterpart. The approximation argument is completed after the a priori estimates in \secref{sec:main-proof}.

\begin{corollary}[Finite-regularity extension]\label{cor:finite-regularity}
Let \(m\ge2\) be an integer. Assume that \(f_0\) is compactly supported in \(\mathbb R^2\) and
\[
u_0\in W^{m,\infty}(\mathbb R),
\qquad
0\le f_0\in W^{m,\infty}(\mathbb R^2).
\eqtaglabel{eq:2.16}
\]
Define
\[
\rho_0(x)=\int_{\mathbb R}f_0(x,v)\,dv,
\qquad
F_0(x)=u_0'(x)+\rho_0(x).
\]
Since \(f_0\) is compactly supported, differentiation under the velocity integral shows that
\(\rho_0\in W^{m,\infty}(\mathbb R)\). Since \(m\ge2\), both \(u_0'\) and
\(\rho_0\) may be taken to be Lipschitz continuous. Hence \(F_0(x)\) is
defined for every \(x\in\mathbb R\).
Then problem \eqref{eq:1.1}--\eqref{eq:1.2} admits a unique global finite-regularity classical solution if and only if
\[
F_0(x)\ge0,
\qquad x\in\mathbb R.
\eqtaglabel{eq:2.17}
\]
More precisely, for every finite \(T>0\),
\[
\begin{aligned}
u&\in C\bigl([0,T];W^{m-1,\infty}(\mathbb R)\bigr)
\cap L^\infty\bigl(0,T;W^{m,\infty}(\mathbb R)\bigr),\\
f&\in C\bigl([0,T];W^{m-1,\infty}(\mathbb R^2)\bigr)
\cap L^\infty\bigl(0,T;W^{m,\infty}(\mathbb R^2)\bigr),
\end{aligned}
\eqtaglabel{eq:2.18}
\]
Moreover, for every finite \(T>0\), the phase-space support of \(f(t)\) remains
in a fixed compact set for \(0\le t\le T\). If \(F_0(x_0)<0\) for some
\(x_0\in\mathbb R\), then the maximal lifespan \(T_{\mathrm{max}}^{(m)}\)
in the class \eqref{eq:2.18} satisfies
\[
T_{\mathrm{max}}^{(m)}
\le\frac{1}{-F_0(x_0)}<\infty.
\eqtaglabel{eq:2.19}
\]
\end{corollary}

\section{Characteristics and the Key Evolution Equation}\label{sec:key}

Let \((u,f)\) be the smooth solution on its maximal interval \([0,T_{\text{max}})\). For this solution, write
\[
\rho(t,x)=\int_{\mathbb R}f(t,x,v)\,dv,\qquad
j(t,x)=\int_{\mathbb R}vf(t,x,v)\,dv.
\eqtaglabel{eq:3.1}
\]
The fluid equation becomes
\[
u_t+uu_x=j-u\rho.
\eqtaglabel{eq:3.2}
\]

The particle characteristics \((X(t),V(t))\) satisfy
\[
\begin{cases}
X'(t)=V(t),\\
V'(t)=u(t,X(t))-V(t),
\end{cases}
\qquad
(X(0),V(0))=(x_0,v_0).
\eqtaglabel{eq:3.3}
\]
Writing the Vlasov equation as
\[
f_t+vf_x+(u-v)f_v=f,
\]
we obtain
\[
\frac{d}{dt}f(t,X(t),V(t))
=f(t,X(t),V(t)),
\]
and hence
\[
f(t,X(t),V(t))=e^tf_0(x_0,v_0).
\eqtaglabel{eq:3.4}
\]
Since \(f_0\ge0\), the characteristic representation \eqref{eq:3.4} gives
\[
f(t,x,v)\ge0,\qquad
\|f(t)\|_{L^\infty}
\le e^t\|f_0\|_{L^\infty}.
\eqtaglabel{eq:3.5}
\]
Integrating the second equation in \eqref{eq:1.1} with respect to \(v\), and using the compact velocity support of \(f(t,x,\cdot)\) on every local interval of smooth existence, yields
\[
\rho_t+j_x=0.
\eqtaglabel{eq:3.6}
\]
Differentiating \eqref{eq:3.2} with respect to \(x\), we find
\[
u_{xt}+u_x^2+uu_{xx}
=j_x-\rho u_x-u\rho_x.
\eqtaglabel{eq:3.7}
\]
To derive the blow-up criterion, we introduce the key quantity that combines the fluid velocity gradient and the particle density:
\[
F=u_x+\rho.
\eqtaglabel{eq:3.8}
\]
Combining \eqref{eq:3.6} and \eqref{eq:3.7} gives
\[
\begin{aligned}
F_t+uF_x
&=(u_{xt}+\rho_t)+u(u_{xx}+\rho_x)\\
&=-u_x^2-\rho u_x\\
&=-u_xF
=F(\rho-F).
\end{aligned}
\]
Therefore,
\[
F_t+uF_x=F(\rho-F).
\eqtaglabel{eq:3.9}
\]

For each \(y_0\in\mathbb R\), let \(Y(t;y_0)\) be the fluid characteristic issued from \(y_0\):
\[
Y'(t;y_0)=u(t,Y(t;y_0)),\qquad
Y(0;y_0)=y_0.
\eqtaglabel{eq:3.10}
\]
By the chain rule and \eqref{eq:3.9},
\[
\begin{aligned}
\frac{d}{dt}F(t,Y(t;y_0))
&=F_t(t,Y(t;y_0))
+Y'(t;y_0)F_x(t,Y(t;y_0))\\
&=(F_t+uF_x)(t,Y(t;y_0)),
\end{aligned}
\]
so that
\[
\frac{d}{dt}F(t,Y(t;y_0))
=F(t,Y(t;y_0))
\bigl(\rho(t,Y(t;y_0))-F(t,Y(t;y_0))\bigr).
\eqtaglabel{eq:3.11}
\]

\section{Proof of the Main Theorem}\label{sec:main-proof}

\subsection{Necessity}\label{subsec:necessity}

Assume that \(F_0(x_0)<0\) for some \(x_0\in\mathbb R\). Since \(F\equiv0\) is an equilibrium solution of \eqref{eq:3.11}, uniqueness for the scalar ordinary differential equation implies that
\[
F(t,Y(t;x_0))<0
\]
throughout the smooth lifespan. Along this characteristic, with the arguments suppressed, \eqref{eq:3.11} and \(\rho\ge0\) give
\[
\begin{aligned}
\frac{d}{dt}\left(-\frac1F\right)
&=\frac{F'}{F^2}
=\frac{F(\rho-F)}{F^2}
=\frac{\rho}{F}-1
\le-1.
\end{aligned}
\]
Hence, for every \(0\le t<T_{\text{max}}\),
\[
0< -\frac1{F(t,Y(t;x_0))}
\le -\frac1{F_0(x_0)}-t.
\]
The left-hand side is positive, so the right-hand side must also be positive. Thus every \(t<T_{\text{max}}\) satisfies
\[
t<-\frac1{F_0(x_0)},
\]
and consequently
\[
T_{\text{max}}
\le\frac{1}{-F_0(x_0)}<\infty.
\]
Therefore the maximal smooth lifespan is finite, and a global smooth solution cannot exist. If equality holds in this lifespan bound, then the preceding reciprocal estimate implies
\[
F(t,Y(t;x_0))\longrightarrow-\infty
\qquad\text{as }t\to T_{\mathrm{max}}^{-}.
\]
Since \(u_x=F-\rho\le F\), it follows that
\[
u_x(t,Y(t;x_0))\longrightarrow-\infty.
\]
Thus, in the equality case, breakdown occurs through negative blow-up of the fluid velocity gradient. This proves the necessity.

\subsection{Sufficiency}\label{subsec:sufficiency}

Assume that
\[
F_0(x)=u_0'(x)+\rho_0(x)\ge0,
\qquad x\in\mathbb R.
\eqtaglabel{eq:4.1}
\]
We establish successively bounds for the zeroth-order quantities, \(F\), the first derivatives, and all higher derivatives on an arbitrary finite time interval, and then apply \propref{prop:local}. Recall the constant \(R_0\) introduced in the proof of \propref{prop:local}. In particular,
\[
\|u_0\|_{L^\infty}\le R_0,
\qquad
|v|\le R_0
\quad\text{on }\operatorname{supp}f_0.
\eqtaglabel{eq:4.2}
\]
We first control the fluid velocity and the particle velocity support. Fix \(T<T_{\text{max}}\). We claim that the interval \([-R_0,R_0]\) is invariant for both velocities on \([0,T]\). By \eqref{eq:3.5}, \(f\) remains nonnegative throughout the smooth lifespan. Along a fluid characteristic \(Y(t;y_0)\), set
\[
U(t)=u(t,Y(t;y_0)).
\]
The fluid equation gives
\[
U'(t)=\int_{\mathbb R}
f(t,Y(t;y_0),v)\bigl(v-U(t)\bigr)\,dv.
\eqtaglabel{eq:4.3}
\]
To avoid assuming that the spatial supremum of \(u\) is attained, fix \(\varepsilon>0\) and set \(R_\varepsilon=R_0+\varepsilon\). On \([0,T]\), \propref{prop:local} confines the particle support to a compact set. A fluid characteristic that does not meet this set satisfies \(U'=0\); hence any outward crossing must occur within a compact family of interacting fluid or particle characteristics, and a first contact is attained. Suppose that, prior to this first contact with the boundary of \([-R_\varepsilon,R_\varepsilon]\), all fluid and particle velocities lie in this interval. At the upper fluid boundary,
\[
U=R_\varepsilon
\quad\Longrightarrow\quad
U'=\int_{\mathbb R}
f(t,Y(t;y_0),v)(v-R_\varepsilon)\,dv\le0,
\]
whereas at the lower fluid boundary,
\[
U=-R_\varepsilon
\quad\Longrightarrow\quad
U'=\int_{\mathbb R}
f(t,Y(t;y_0),v)(v+R_\varepsilon)\,dv\ge0.
\]
Thus the fluid velocity cannot cross either boundary outward. Similarly, along a particle characteristic, \(V'=u(t,X)-V\), and hence
\[
V=R_\varepsilon
\quad\Longrightarrow\quad
V'=u(t,X)-R_\varepsilon\le0,
\qquad
V=-R_\varepsilon
\quad\Longrightarrow\quad
V'=u(t,X)+R_\varepsilon\ge0.
\eqtaglabel{eq:4.4}
\]
The particle velocity cannot cross either boundary outward. Since all initial fluid and particle velocities lie strictly inside \([-R_\varepsilon,R_\varepsilon]\), the standard first-contact argument gives
\[
|u(t,x)|\le R_\varepsilon,
\qquad
|V(t)|\le R_\varepsilon,
\qquad 0\le t\le T.
\]
Letting \(\varepsilon\to0\), using that the fluid characteristic flow covers \(\mathbb R\), and applying the representation \eqref{eq:3.4} to particle characteristics issued from \(\operatorname{supp}f_0\), we obtain
\[
\|u(t)\|_{L^\infty}\le R_0,\qquad
f(t,x,v)=0\quad\text{for }|v|>R_0,
\qquad 0\le t\le T.
\eqtaglabel{eq:4.5}
\]
Since \(T<T_{\text{max}}\) is arbitrary, \eqref{eq:4.5} holds throughout the maximal smooth lifespan.

By the same velocity-moment calculation as in \propref{prop:local}, now using \eqref{eq:3.5} and the bound \(|v|\le R_0\) from \eqref{eq:4.5}, for \(0\le t\le T\) we obtain
\[
\|\rho(t)\|_{L^\infty}
\le2R_0e^T\|f_0\|_{L^\infty}.
\eqtaglabel{eq:4.6}
\]
Similarly,
\[
\|j(t)\|_{L^\infty}
\le R_0^2e^T\|f_0\|_{L^\infty}.
\eqtaglabel{eq:4.7}
\]

For each \(y_0\in\mathbb R\), equation \eqref{eq:3.11} is a scalar ordinary differential equation for \(t\mapsto F(t,Y(t;y_0))\). Its right-hand side is locally Lipschitz in \(F\), and \(F\equiv0\) is an equilibrium solution. Since \(F_0(y_0)\ge0\), uniqueness prevents \(F(t,Y(t;y_0))\) from crossing zero. As the fluid characteristic flow covers \(\mathbb R\), condition \eqref{eq:4.1} therefore implies
\[
F(t,x)\ge0.
\eqtaglabel{eq:4.8}
\]
Moreover, \eqref{eq:4.6} gives, along every fluid characteristic,
\[
\frac{dF}{dt}
=F(\rho-F)
\le
F\left(2R_0e^T\|f_0\|_{L^\infty}-F\right).
\eqtaglabel{eq:4.9}
\]
Set
\[
M_T:=\max\left\{
\|F_0\|_{L^\infty},
2R_0e^T\|f_0\|_{L^\infty}
\right\}.
\]
By \eqref{eq:4.6},
\[
0\le\rho(t,x)\le2R_0e^T\|f_0\|_{L^\infty},
\qquad 0\le t\le T.
\]
The lower boundary is invariant because \(F=0\) is an equilibrium solution of \eqref{eq:3.11}. At the upper boundary \(F=M_T\),
\[
\frac{dF}{dt}
=M_T(\rho-M_T)
\le M_T\left(
2R_0e^T\|f_0\|_{L^\infty}-M_T
\right)
\le0.
\]
Therefore the interval \([0,M_T]\) is forward invariant along every fluid characteristic, and hence
\[
0\le F(t,x)\le M_T,
\qquad 0\le t\le T.
\eqtaglabel{eq:4.10}
\]
Since \(u_x=F-\rho\),
\[
\begin{aligned}
\|u_x(t)\|_{L^\infty}
&=\|F(t)-\rho(t)\|_{L^\infty}\\
&\le\|F(t)\|_{L^\infty}+\|\rho(t)\|_{L^\infty}\\
&\le M_T+2R_0e^T\|f_0\|_{L^\infty}\\
&\le\|F_0\|_{L^\infty}
+4R_0e^T\|f_0\|_{L^\infty}
\le C_T,
\qquad 0\le t\le T.
\end{aligned}
\eqtaglabel{eq:4.11}
\]
Here and below, \(C_T\) denotes a finite positive constant depending only on \(T\), \(R_0\), \(\|F_0\|_{L^\infty}\), and \(\|f_0\|_{C_b^1}\). Its value may change from line to line, but it is independent of \(t\in[0,T]\).

We next estimate the first derivatives of the particle distribution \(f\). Write the Vlasov equation as
\[
f_t+vf_x+(u-v)f_v=f.
\]
Differentiation with respect to \(x\) and \(v\) gives
\[
\begin{aligned}
(f_x)_t+v(f_x)_x+(u-v)(f_x)_v
&=f_x-u_xf_v,\\
(f_v)_t+v(f_v)_x+(u-v)(f_v)_v
&=2f_v-f_x.
\end{aligned}
\eqtaglabel{eq:4.12}
\]
Taking absolute values, adding the two inequalities, and integrating along particle characteristics, we obtain
\[
\frac{d}{dt}\bigl(|f_x|+|f_v|\bigr)(t,X(t),V(t))
\le
\bigl(2+\|u_x(t)\|_{L^\infty}\bigr)
\bigl(|f_x|+|f_v|\bigr)(t,X(t),V(t)).
\]
Thus \eqref{eq:4.11} and Grönwall's inequality imply
\[
\begin{aligned}
&\|f_x(t)\|_{L^\infty}+\|f_v(t)\|_{L^\infty}\\
&\quad\le
\left(\|(f_0)_x\|_{L^\infty}+\|(f_0)_v\|_{L^\infty}\right)
\exp\!\left(
\int_0^t\bigl(2+\|u_x(\tau)\|_{L^\infty}\bigr)\,d\tau
\right)
\le C_T.
\end{aligned}
\eqtaglabel{eq:4.13}
\]
Therefore all first-order spatial and velocity derivatives of \(u\) and \(f\) are bounded on every finite time interval. For \(m\ge1\), set
\[
H_m(t):=
1+\|u(t)\|_{C_b^m(\mathbb R)}
+\|f(t)\|_{C_b^m(\mathbb R^2)}.
\]
The preceding estimates show that \(\sup_{0\le t\le T}H_1(t)<\infty\). We now prove the corresponding bound for every \(m\ge2\). Assume inductively that
\[
\sup_{0\le t\le T}H_{m-1}(t)<\infty.
\]

For \(p+q=m\), Leibniz's rule gives
\[
\partial_x^p\partial_v^q(vf_x)
=v(\partial_x^p\partial_v^qf)_x
+q\,\partial_x^{p+1}\partial_v^{q-1}f
\]
and
\[
\begin{aligned}
\partial_x^p\partial_v^q\bigl((u-v)f_v\bigr)
={}&(u-v)(\partial_x^p\partial_v^qf)_v
-q\,\partial_x^p\partial_v^qf\\
&+\sum_{h=1}^{p}\binom ph
(\partial_x^hu)\partial_x^{p-h}\partial_v^{q+1}f.
\end{aligned}
\]
Here and below, the term containing \(q-1\) is omitted when \(q=0\). Applying \(\partial_x^p\partial_v^q\) to the Vlasov equation and using the preceding identities gives
\[
\begin{aligned}
&(\partial_x^p\partial_v^qf)_t
+v(\partial_x^p\partial_v^qf)_x
+(u-v)(\partial_x^p\partial_v^qf)_v\\
&\quad=(q+1)\partial_x^p\partial_v^qf
-q\partial_x^{p+1}\partial_v^{q-1}f
-\sum_{h=1}^{p}\binom ph
(\partial_x^hu)\partial_x^{p-h}\partial_v^{q+1}f,
\end{aligned}
\eqtaglabel{eq:4.14}
\]
For the indices occurring in the sum, the product terms can be classified as follows:
\[
\begin{aligned}
h=1:\quad
&\left\|p\,u_x\,\partial_x^{p-1}\partial_v^{q+1}f\right\|_{L^\infty}
\le C_m\|u_x\|_{L^\infty}\|f\|_{C_b^m},\\[1mm]
2\le h\le m-1:\quad
&\left\|(\partial_x^hu)
\partial_x^{p-h}\partial_v^{q+1}f\right\|_{L^\infty}
\le C_{m,T},\\[1mm]
h=m:\quad
&\left\|(\partial_x^mu)f_v\right\|_{L^\infty}
\le\|f_v\|_{L^\infty}\|u\|_{C_b^m}.
\end{aligned}
\]
The last case occurs only when \((p,q,h)=(m,0,m)\). The middle terms contain only derivatives of order at most \(m-1\) and are therefore controlled by the induction hypothesis. Consequently,
\[
\begin{aligned}
&\left\|
\sum_{h=1}^{p}\binom ph
(\partial_x^hu)\partial_x^{p-h}\partial_v^{q+1}f
\right\|_{L^\infty}\\
&\qquad\le
C_m\|u_x\|_{L^\infty}\|f\|_{C_b^m}
+C_m\|f_v\|_{L^\infty}\|u\|_{C_b^m}
+C_{m,T}.
\end{aligned}
\]
Thus every derivative of order \(m\) occurs linearly and is multiplied only by a quantity already controlled by the induction hypothesis. Integration of \eqref{eq:4.14} along particle characteristics, followed by summation over \(p+q\le m\), yields
\[
\|f(t)\|_{C_b^m}
\le\|f_0\|_{C_b^m}
+C_{m,T}\int_0^t
\left(
1+\|f(\tau)\|_{C_b^m}
+\|u(\tau)\|_{C_b^m}
\right)\,d\tau.
\eqtaglabel{eq:4.15}
\]
Here \(C_{m,T}\) depends only on \(m\), \(R_0\), and the already established bound for \(H_{m-1}\) on \([0,T]\); in particular, it is independent of the \(m\)-th order quantities still being estimated.

For the fluid equation, Leibniz's rule gives
\[
\partial_x^m(uu_x)
=u\,\partial_x^{m+1}u
+\sum_{i=1}^{m}\binom mi
(\partial_x^iu)(\partial_x^{m-i+1}u)
\]
and
\[
\partial_x^m(u\rho)
=u\,\partial_x^m\rho+\rho\,\partial_x^mu
+\sum_{i=1}^{m-1}\binom mi
(\partial_x^iu)(\partial_x^{m-i}\rho).
\]
The terms corresponding to \(i=1\) and \(i=m\) in the first sum combine to \((m+1)u_x\partial_x^mu\). Consequently,
\[
\begin{aligned}
(\partial_x^mu)_t+u(\partial_x^mu)_x
={}&\partial_x^mj-u\,\partial_x^m\rho
-\bigl(\rho+(m+1)u_x\bigr)\partial_x^mu\\
&-\sum_{i=1}^{m-1}\binom mi
(\partial_x^iu)(\partial_x^{m-i}\rho)\\
&-\sum_{i=2}^{m-1}\binom mi
(\partial_x^iu)(\partial_x^{m-i+1}u),
\end{aligned}
\eqtaglabel{eq:4.16}
\]
where the last sum is empty when \(m=2\). By the velocity-support bound and the velocity-moment calculation in \propref{prop:local},
\[
\|\partial_x^m\rho(t)\|_{L^\infty}
\le2R_0\|\partial_x^mf(t)\|_{L^\infty},
\qquad
\|\partial_x^mj(t)\|_{L^\infty}
\le R_0^2\|\partial_x^mf(t)\|_{L^\infty}.
\eqtaglabel{eq:4.17}
\]
The terms on the right-hand side of \eqref{eq:4.16} are estimated separately. First, \eqref{eq:4.5} and \eqref{eq:4.17} give
\[
\begin{aligned}
\|\partial_x^mj-u\,\partial_x^m\rho\|_{L^\infty}
&\le\|\partial_x^mj\|_{L^\infty}
+\|u\|_{L^\infty}\|\partial_x^m\rho\|_{L^\infty}\\
&\le\left(R_0^2+2R_0\|u\|_{L^\infty}\right)
\|\partial_x^mf\|_{L^\infty}\\
&\le C_{R_0}\|f\|_{C_b^m}.
\end{aligned}
\]
Next, the coefficient of the top-order fluid derivative is already controlled:
\[
\left\|\bigl(\rho+(m+1)u_x\bigr)\partial_x^mu\right\|_{L^\infty}
\le
\left(\|\rho\|_{L^\infty}
+(m+1)\|u_x\|_{L^\infty}\right)
\|u\|_{C_b^m}.
\]
Finally, every factor in the two sums in \eqref{eq:4.16} has order at most \(m-1\), since
\[
\begin{aligned}
1\le i\le m-1
&\quad\Longrightarrow\quad i\le m-1,
\qquad m-i\le m-1,\\
2\le i\le m-1
&\quad\Longrightarrow\quad i\le m-1,
\qquad m-i+1\le m-1.
\end{aligned}
\]
The induction hypothesis and the velocity-moment estimate therefore imply
\[
\begin{aligned}
&\left\|
\sum_{i=1}^{m-1}\binom mi
(\partial_x^iu)(\partial_x^{m-i}\rho)
\right\|_{L^\infty}\\
&\quad+
\left\|
\sum_{i=2}^{m-1}\binom mi
(\partial_x^iu)(\partial_x^{m-i+1}u)
\right\|_{L^\infty}
\le C_{m,T}.
\end{aligned}
\]
Integrating \eqref{eq:4.16} along fluid characteristics now gives
\[
\|u(t)\|_{C_b^m}
\le\|u_0\|_{C_b^m}
+C_{m,T}\int_0^t
\left(
1+\|u(\tau)\|_{C_b^m}
+\|f(\tau)\|_{C_b^m}
\right)\,d\tau.
\eqtaglabel{eq:4.18}
\]
Adding \eqref{eq:4.15} and \eqref{eq:4.18}, we obtain
\[
H_m(t)\le H_m(0)+C_{m,T}\int_0^tH_m(\tau)\,d\tau.
\]
Grönwall's inequality therefore yields
\[
\sup_{0\le t\le T}H_m(t)
\le H_m(0)e^{C_{m,T}T}<\infty.
\eqtaglabel{eq:4.19}
\]
Starting from \(m=2\) and arguing inductively proves that all finite-order spatial and velocity derivatives remain bounded on \([0,T]\).

We now complete the continuation argument. Suppose, for contradiction, that
\[
T_{\text{max}}<\infty.
\]
The constants in the preceding a priori estimates depend only on the initial data and on a prescribed finite upper bound for time. Therefore, under the present assumption, these constants can be chosen uniformly for all \(T<T_{\text{max}}\). Taking \(m=2\) in \eqref{eq:4.19}, we obtain
\[
\sup_{0\le t<T_{\text{max}}}
\left(
\|u(t)\|_{C_b^2(\mathbb R)}
+\|f(t)\|_{C_b^2(\mathbb R^2)}
\right)<\infty.
\]
This contradicts the continuation criterion \eqref{eq:2.4}. Hence
\[
T_{\text{max}}=\infty.
\]
Since \(T>0\) is arbitrary, \eqref{eq:4.19} further gives
\[
\sup_{0\le t\le T}H_m(t)<\infty
\qquad
\text{for every finite \(T>0\) and every integer \(m\ge2\)}.
\]
Thus all finite-order spatial and velocity derivatives remain bounded on every finite time interval. The corresponding time regularity follows directly from the equations. Consequently, the solution is globally smooth. This proves the sufficiency and completes the proof of \thmref{thm:main}. \(\square\)

\subsection{Finite-regularity data}\label{subsec:finite-regularity}

We finish by proving \corref{cor:finite-regularity}. The argument follows the smooth-approximation strategy used for general initial data in \cite{ref10}. We first construct a local solution for arbitrary data satisfying \eqref{eq:2.16}. We then use a nonnegative regularization, which preserves the threshold condition, to obtain global existence when \(F_0\ge0\).

Let \(\eta\in C_c^\infty(\mathbb R)\) be nonnegative and satisfy
\[
\int_{\mathbb R}\eta(z)\,dz=1.
\]
For \(\varepsilon>0\), define
\[
\eta_\varepsilon(z)
=
\frac{1}{\varepsilon}
\eta\!\left(\frac{z}{\varepsilon}\right).
\]
For \(0<\varepsilon\le1\), set
\[
u_0^\varepsilon=\eta_\varepsilon*u_0,
\qquad
f_0^\varepsilon(x,v)
=\iint_{\mathbb R^2}
\eta_\varepsilon(x-y)\eta_\varepsilon(v-w)
f_0(y,w)\,dy\,dw.
\eqtaglabel{eq:4.20}
\]
Then
\[
u_0^\varepsilon\in C_b^\infty(\mathbb R),
\qquad
0\le f_0^\varepsilon\in C_c^\infty(\mathbb R^2),
\]
and the supports of \(f_0^\varepsilon\), \(0<\varepsilon\le1\), are contained
in one compact set. The following calculation is written on
\(\mathbb R^d\), with \(d=1\) for \(u_0\) and \(d=2\) for \(f_0\); in the
latter case, \(\eta_\varepsilon\) denotes the product kernel appearing in
\eqref{eq:4.20}. Differentiation under the convolution and Young's
inequality give
\[
\left\|\partial^\alpha(\eta_\varepsilon*g)\right\|_{L^\infty}
=
\left\|\eta_\varepsilon*\partial^\alpha g\right\|_{L^\infty}
\le
\left\|\partial^\alpha g\right\|_{L^\infty},
\qquad |\alpha|\le m.
\]
Hence
\[
\begin{aligned}
&\sup_{0<\varepsilon\le1}
\left(
\|u_0^\varepsilon\|_{W^{m,\infty}}
+\|f_0^\varepsilon\|_{W^{m,\infty}}
\right)
\le
\|u_0\|_{W^{m,\infty}}
+\|f_0\|_{W^{m,\infty}},\\
&u_0^\varepsilon\longrightarrow u_0
\quad\text{in }W^{m-1,\infty}(\mathbb R),
\qquad
f_0^\varepsilon\longrightarrow f_0
\quad\text{in }W^{m-1,\infty}(\mathbb R^2).
\end{aligned}
\eqtaglabel{eq:4.21}
\]
For completeness, we verify the strong convergence in \eqref{eq:4.21}.
Let \(|\alpha|\le m-1\). Then
\[
\partial^\alpha(\eta_\varepsilon*g-g)(x)
=\int_{\mathbb R^d}\eta_\varepsilon(y)
\bigl(\partial^\alpha g(x-y)-\partial^\alpha g(x)\bigr)\,dy.
\]
Since \(|\alpha|+1\le m\), the function \(\partial^\alpha g\) is Lipschitz
continuous. Therefore,
\[
\begin{aligned}
\left\|\partial^\alpha(\eta_\varepsilon*g-g)\right\|_{L^\infty}
&\le
\|\nabla\partial^\alpha g\|_{L^\infty}
\int_{\mathbb R^d}|y|\eta_\varepsilon(y)\,dy\\
&\le
C_\eta\varepsilon\|g\|_{W^{m,\infty}}.
\end{aligned}
\]
Summing over \(|\alpha|\le m-1\) gives
\[
\|\eta_\varepsilon*g-g\|_{W^{m-1,\infty}}
\le C_\eta\varepsilon\|g\|_{W^{m,\infty}}.
\]
Applying the preceding estimate with \(d=2\) to the product kernel in
\eqref{eq:4.20}, we obtain
\[
\|f_0^\varepsilon-f_0\|_{W^{m-1,\infty}(\mathbb R^2)}
\le
C_\eta\varepsilon
\|f_0\|_{W^{m,\infty}(\mathbb R^2)}.
\]
Hence \(f_0^\varepsilon\to f_0\) strongly in
\(W^{m-1,\infty}(\mathbb R^2)\). Notice that these estimates are stated
one order below \(m\): applying the same increment argument to a derivative
of order \(m\) would require a derivative of order \(m+1\), which is not
assumed in \eqref{eq:2.16}.

The construction in \propref{prop:local} and its fixed-order estimates give
a time \(T_*>0\), independent of \(\varepsilon\), on which all the smooth
approximate solutions exist. This common interval will be used to construct
a local finite-regularity solution for arbitrary data satisfying
\eqref{eq:2.16}. We next verify that the threshold condition is preserved
by the regularization.

Let
\[
\rho_0^\varepsilon(x)
=\int_{\mathbb R}f_0^\varepsilon(x,v)\,dv.
\]
Using the definition of \(f_0^\varepsilon\), Fubini's theorem, and
\(\int_{\mathbb R}\eta_\varepsilon(v-w)\,dv=1\), we obtain
\[
\begin{aligned}
\rho_0^\varepsilon(x)
&=\int_{\mathbb R}\iint_{\mathbb R^2}
\eta_\varepsilon(x-y)\eta_\varepsilon(v-w)
f_0(y,w)\,dy\,dw\,dv\\
&=\int_{\mathbb R}\eta_\varepsilon(x-y)
\left(\int_{\mathbb R}f_0(y,w)\,dw\right)dy\\
&=(\eta_\varepsilon*\rho_0)(x).
\end{aligned}
\]
Consequently,
\[
\begin{aligned}
F_0^\varepsilon
&=(u_0^\varepsilon)'+\rho_0^\varepsilon\\
&=\eta_\varepsilon*(u_0'+\rho_0)
=\eta_\varepsilon*F_0.
\end{aligned}
\eqtaglabel{eq:4.22}
\]
If \(F_0\ge0\), then \(\eta_\varepsilon\ge0\) and \eqref{eq:4.22} imply \(F_0^\varepsilon\ge0\). In this case, \thmref{thm:main} gives a global smooth approximate solution for every \(\varepsilon>0\).

We now fix the time interval on which the compactness argument is performed.
For arbitrary data satisfying \eqref{eq:2.16}, we work on the common local
existence interval \([0,T_*]\). If \(F_0\ge0\), all the smooth approximate
solutions are global, and the same argument can be carried out on any
prescribed finite interval \([0,T]\). On the chosen interval, the velocity
and support estimates in \secref{sec:main-proof}, followed by the fixed-order
estimates leading to \eqref{eq:4.19}, depend only on the length of the
interval, the common support bound, and the initial norms of order at most
\(m\). Hence \eqref{eq:4.21} gives a constant \(C_{m,T}\), independent of
\(\varepsilon\), such that
\[
\sup_{0<\varepsilon\le1}
\sup_{0\le t\le T}
\left(
\|u^\varepsilon(t)\|_{W^{m,\infty}}
+\|f^\varepsilon(t)\|_{W^{m,\infty}}
\right)
\le C_{m,T}.
\eqtaglabel{eq:4.23}
\]
The phase-space supports of \(f^\varepsilon(t)\) are also contained in one compact set independent of \(\varepsilon\).

We next verify strong convergence one derivative below the uniform bound. For \(0<\varepsilon,\delta\le1\), set
\[
\begin{aligned}
D_{\varepsilon,\delta}(t)
:=\sup_{0\le s\le t}\Bigl(&
\|u^\varepsilon(s)-u^\delta(s)\|_{W^{m-1,\infty}}
\\
&+\|f^\varepsilon(s)-f^\delta(s)\|_{W^{m-1,\infty}}
\Bigr).
\end{aligned}
\]
Write
\[
\delta u=u^\varepsilon-u^\delta,
\quad
\delta f=f^\varepsilon-f^\delta,
\quad
\delta\rho=\rho^\varepsilon-\rho^\delta,
\quad
\delta j=j^\varepsilon-j^\delta.
\]
Subtracting the two particle equations and using the transport field generated by \(u^\varepsilon\), we obtain
\[
(\delta f)_t+v(\delta f)_x
+(u^\varepsilon-v)(\delta f)_v
=\delta f-\delta u\,f_v^\delta.
\eqtaglabel{eq:4.24}
\]
Similarly, the two fluid equations give
\[
(\delta u)_t+u^\varepsilon(\delta u)_x
=\delta j-\rho^\varepsilon\delta u
-u^\delta\delta\rho-\delta u\,u_x^\delta.
\eqtaglabel{eq:4.25}
\]
Both differences have initial values determined by \eqref{eq:4.21}. Differentiating \eqref{eq:4.24}--\eqref{eq:4.25} up to order \(m-1\), and integrating along the corresponding characteristics, give
\[
D_{\varepsilon,\delta}(t)
\le
C_{m,T}D_{\varepsilon,\delta}(0)
+C_{m,T}\int_0^t
D_{\varepsilon,\delta}(\tau)\,d\tau,
\qquad 0\le t\le T.
\eqtaglabel{eq:4.26}
\]
To justify the closure of \eqref{eq:4.26}, observe that every derivative falling on a difference has order at most \(m-1\), whereas the remaining coefficient is bounded by \eqref{eq:4.23}; if one derivative falls on such a coefficient, at most \(m\) derivatives of \(u^\varepsilon\) or \(f^\varepsilon\) occur. The moment differences are bounded, uniformly in \(\varepsilon\) and \(\delta\), by the corresponding derivatives of \(f^\varepsilon-f^\delta\), because all velocity supports lie in the same compact set. This is the same finite-order difference estimate as in \eqref{eq:2.13}, applied at order \(m-1\).

By Grönwall's inequality and \eqref{eq:4.21}, the family \((u^\varepsilon,f^\varepsilon)\) is Cauchy in
\[
C\bigl([0,T];W^{m-1,\infty}(\mathbb R)
\times W^{m-1,\infty}(\mathbb R^2)\bigr).
\]
Denote the strong limit by \((u,f)\). By the uniform estimate \eqref{eq:4.23} and the Banach--Alaoglu theorem, the derivatives of order \(m\) are weak-* compact in \(L^\infty\). On the other hand, the strong convergence in \(W^{m-1,\infty}\) implies convergence in the sense of distributions. Since differentiation is continuous in the space of distributions, the weak-* limits are precisely the corresponding \(m\)-th order distributional derivatives of \(u\) and \(f\). Consequently,
\[
(u,f)\in
L^\infty\bigl(0,T;W^{m,\infty}(\mathbb R)
\times W^{m,\infty}(\mathbb R^2)\bigr).
\]
Moreover, the estimate \eqref{eq:4.23} passes to the limit.
Since \(m\ge2\), the strong convergence takes place at least in \(W^{1,\infty}\). The nonlinear products therefore converge strongly, while the velocity moments converge strongly because of the common compact velocity support. Passing to the limit in the two equations proves that \((u,f)\) solves \eqref{eq:1.1}--\eqref{eq:1.2}. The equations and the uniform \(W^{m,\infty}\) bound give
\[
u_t\in L^\infty\bigl(0,T;W^{m-1,\infty}(\mathbb R)\bigr),
\qquad
f_t\in L^\infty\bigl(0,T;W^{m-1,\infty}(\mathbb R^2)\bigr).
\]
They also give the strong time continuity in \(W^{m-1,\infty}\), and therefore \eqref{eq:2.18} holds on \([0,T]\). Without a sign condition on \(F_0\), this argument with \(T=T_*\) constructs a local finite-regularity solution. If \(F_0\ge0\), then \(T>0\) is arbitrary, so the solution is global. The estimate \eqref{eq:4.26}, now applied to two solutions with identical initial data, proves uniqueness in this class.

We have proved global existence when the initial threshold is nonnegative. We now prove necessity. Suppose that \(F_0(x_0)<0\) for some \(x_0\in\mathbb R\), and let \([0,T_{\mathrm{max}}^{(m)})\) be the maximal interval of existence of the corresponding finite-regularity solution. Since \(m\ge2\), the function \(F=u_x+\rho\) is locally Lipschitz in the spatial variable, while the transport identity \eqref{eq:3.9} holds almost everywhere. The chain rule along the fluid characteristic \(Y(t;x_0)\) therefore shows that \eqref{eq:3.11} remains valid for almost every \(t<T_{\mathrm{max}}^{(m)}\). Hence the sign-preservation and reciprocal argument used in the necessity part of \thmref{thm:main} applies without change and gives
\[
T_{\mathrm{max}}^{(m)}
\le\frac{1}{-F_0(x_0)}<\infty.
\]
Thus no global solution in the class \eqref{eq:2.18} can exist when the threshold fails. This completes the proof of \corref{cor:finite-regularity}. \(\square\)

\end{document}